\documentclass[12pt]{amsart}
\usepackage[english]{babel}
\usepackage{amssymb, amsmath, amsthm, amsfonts}
\usepackage[]{amsrefs}
\usepackage[all]{xy}
\usepackage{tabulary,caption,comment}
\usepackage{xcolor}
\usepackage[unicode,pdfencoding=auto,psdextra]{hyperref}

\theoremstyle{plain}

\newtheorem{Thm}{Theorem}[section]
\newtheorem{Cor}[Thm]{Corollary}
\newtheorem{Lem}[Thm]{Lemma}
\newtheorem{Prop}[Thm]{Proposition}

\newtheorem{OP}[Thm]{Open Problem}

\theoremstyle{definition}
\newtheorem{Def}[Thm]{Definition}

\theoremstyle{remark}
\newtheorem{Rem}[Thm]{Remark}

\newcommand{\C}{\mathbb{C}}
\newcommand{\R}{\mathbb{R}}
\newcommand{\N}{\mathbb{N}}

\newcommand{\Aut}{\mathrm{Aut}}
\newcommand{\SAut}{\mathrm{SAut}}
\newcommand{\End}{\mathrm{End}}
\newcommand{\U}{\mathrm{U}}
\renewcommand{\O}{\mathcal{O}}
\newcommand{\bsr}{\mathrm{bsr}}

\makeatletter
\@namedef{subjclassname@2020}{%
  \textup{2020} Mathematics Subject Classification}
\makeatother

\numberwithin{equation}{section}

\title[Parametric Factorization of Matrices]%
{Parametric Factorization of Matrices}
\author{Gaofeng Huang \and Frank Kutzschebauch}

\address{Mathematisches Institut\\
Universit\"at Bern\\
Sidlerstrasse 5, CH--3012 Bern, Switzerland}

\email{gaofeng.huang@unibe.ch}
\email{frank.kutzschebauch@unibe.ch}

\thanks{The research was partially supported by Schweizerische Nationalfonds Grant 200021-178730.}


\begin{document}

\maketitle

\begin{abstract} In this survey paper we study  parametric versions of writing a  matrix in $SL_n (\mathbb{C})$ as a product of lower and upper unitriangular matrices in interchanging order as well as generalizations to other classical groups. We give an account of algebraic, continuous and holomorphic factorization results, from the standpoint of Several Complex Variables. Out of the wealth of algebraic results, we only concentrate on those which are related to holomorphic factorization and often formulate them in a specific form, e.g.\@ for the field of complex numbers in place of more general fields or principal ideal domains. 
The number of unitriangular matrices needed is a difficult problem and 
is solved in very specific  cases only. We give a new lower bound for
factorizing matrices in $SL_2 (\mathbb{C})$ continuously parametrized by normal topological spaces of dimension bigger than one.

\end{abstract}

\section{Introduction}

The well-known LULU decomposition states that every matrix $M= (M_{ij})_{ i,j =1}^n$ in $SL_n(k)$, where $k$ is an arbitrary field, can be written as a product of $4$ unitriangular matrices
\begin{align*}
    M=\begin{pmatrix}
    1 & 0 & \cdots & 0 \\
    * & 1 & \ddots & \vdots \\
    \vdots & \ddots & \ddots & 0 \\
    * & \cdots & * & 1
    \end{pmatrix}
    \begin{pmatrix}
    1 & * & \cdots & * \\
    0 & 1 & \ddots & \vdots \\
    \vdots & \ddots & \ddots & *\\
    0 & \cdots & 0 & 1
    \end{pmatrix}
    \begin{pmatrix}
    1 & 0 & \cdots & 0 \\
    * & 1 & \ddots & \vdots \\
    \vdots & \ddots & \ddots & 0 \\
    * & \cdots & * & 1
    \end{pmatrix}
    \begin{pmatrix}
    1 & * & \cdots & * \\
    0 & 1 & \ddots & \vdots \\
    \vdots & \ddots & \ddots & *\\
    0 & \cdots & 0 & 1
    \end{pmatrix}.
\end{align*}

It is natural to consider algebraic dependence on the parameter, which means that the entries $M_{ij}$ of $M$ are polynomial functions on an affine algebraic $k$-variety $k[X]$ and we look for a unitriangular factorization such that the off-diagonal entries of the factors are elements of $k[X]$ as well. In this generality it is necessary to ask for a factorization of $M \in SL_n(k[X])$ into finitely many (possibly more than $4$) unitriangular factors. One can  of course not only consider  algebraic dependence on a parameter. Our main interest is holomorphic dependence on parameters, therefore  we restrict ourselves to the field of complex numbers $k = \C$. Moreover in our case of main interest  the entries $M_{ij}$'s are holomorphic functions on a Stein space $X$, i.e.,\@ elements of $\mathcal{O}(X)$.   The main tool for solving holomorphic factorization problems turns out to be the Oka principle. Informally speaking an Oka principle holds if one can glue local holomorphic solutions to a global holomorphic solution in the presence of a continuous solution. This naturally leads us to consider continuous dependence on the parameter. Now the entries $M_{ij}$'s are complex-valued continuous functions on a topological space $T$, i.e., elements of $\mathcal{C}(T)$. 

The unifying concept of the above mentioned dependences is to consider the special linear group $SL_n(R)$ over a unital commutative ring $R$ (in our specific cases $R$ was the ring of polynomial functions on an affine variety, the ring of holomorphic functions on a Stein space, and the ring of continuous functions on a topological space resp.\@). Finite products of unitriangular matrices over $R$ form the so-called elementary subgroup $E_n(R)$ and our factorization problem is the same as asking whether $E_n(R) = SL_n(R)$. 

Instead of $SL_n$ one can consider other classical groups. For a semisimple Chevalley group G(R) over a unital commutative ring $R$, the factorization problem asks whether $E(R)$ coincides with $ G(R)$, where $E(R)$ denotes the elementary Chevalley group corresponding to $G(R)$. In fact, results in this generality exist mainly in the setting of algebraic parametric dependence or are consequences of purely K-theoretic properties, see Section \ref{subsec:bsr1}, \ref{subsec:banachalg} and Section \ref{sec: alg}. 
While for the continuous and holomorphic cases where the ring $R$ is $\mathcal{C}(T)$ and $\mathcal{O}(X)$ resp., the only factorization results which do not follow from K-theoretic considerations known until now have been obtained for the special linear groups $SL_n$ and the symplectic groups $Sp_{2n}$.

The first results on the holomorphic factorization date back to about 15 years ago. Although people were aware of the problem earlier, it was explicitly asked by Gromov in 1989. The algebraic factorization problem is part of algebraic K-theory, which was initiated in the late 1950s. Important notions like the Bass stable rank were introduced. The first celebrated factorization result was achieved by Suslin in 1977, as a byproduct of his solution to the Serre conjecture. Modern algebraic results connect algebraic factorization to motivic homotopy theory developed by Morel and Voevodsky. The problem of continuous factorization and the connection to algebraic K-theory had been promoted by Vaserstein since the 1960s. This led him to introducing the topological stable rank and he finally solved the problem in 1988.

In Section \ref{sec: general} we present some results which depend on general algebraic or topological properties of the ring $R$. We give examples of polynomial rings or function rings with these properties, thus showing the relevance of the presented results for the question of parametric factorization. On the other hand, for continuous or holomorphic dependence of parameter, many function rings  do not satisfy these properties. Instead, one should impose the nullhomotopy condition. We discuss the role of nullhomotopy in parametric factorization for continuous, holomorphic and even algebraic dependence in Section \ref{sec: nullhom}. Next, in Section \ref{sec: alg} we give specific results over polynomial rings. In Section \ref{sec: cont} we collect all known results on rings of continuous complex-valued functions. Then in Section \ref{sec :holo} we give known factorizations for rings of holomorphic functions on Stein spaces. When a factorization into unitriangular factors exists, it is natural to ask about the number of factors needed. Finding the optimal number is generally a difficult task. We comment on this for all results from Section \ref{sec: general} to \ref{sec :holo}. Section \ref{sec: VB} is devoted to a generalization of parametric factorization into another direction, the setting of vector bundle automorphisms. We are not aware of any general formulation of the vector bundle setting in terms of algebraic K-theory. In the last section we propose some future development and open problems.

\section{General results} \label{sec: general}
The aim of this section is to present results which do focus on properties of the ring $R$. The Euclidean property
can be applied for polynomial parameter dependence and the results in the other two subsections can be both applied for the continuous and in the holomorphic parameter dependence. We will exemplify this.

\subsection{Euclidean rings} 

It is folklore in $K$-theory that every matrix in $SL_n$ for $n \ge 2$ over a Euclidean ring $R$ can be written as a product of finitely many unitriangular factors. Let us sketch the idea by looking at the case of $SL_2$. Let $ A = \begin{pmatrix}   a & * \\ b & *
    \end{pmatrix}$ be an element of $SL_2(R)$. The ideal generated by the two elements $a, b$ in the first column is the entire ring by virtue of the determinant condition. Let $f \colon R \setminus \{0\} \to \N$ be the Euclidean function for $R$. Then there exist $q_1, r_1 \in R$  such that $a = q_1 b + r_1$ and $f(r_1) < f(b)$. Multiplying $A$ from the left by $\begin{pmatrix}   1 & -q_1 \\ 0 & 1
    \end{pmatrix}$ gives $\begin{pmatrix}   r_1 & * \\ b & *
    \end{pmatrix}$. Division of $b$ by $r_1$ yields $b = q_2 r_1 + r_2$ for $q_2, r_2 \in R$ and $f(r_2)<f(r_1)$. Another left multiplication with $\begin{pmatrix}   1 & 0 \\ -q_2 & 1
    \end{pmatrix}$ gives $\begin{pmatrix}   r_1 & * \\ r_2 & *
    \end{pmatrix}$. After finitely many steps we get $\begin{pmatrix}   \alpha & * \\ \beta & *
    \end{pmatrix}$ where one of $\alpha, \beta$, say $\alpha$, is the greatest common divisor of $a, b$. Since the ideal $(a, b) = (\alpha) = R$, we deduce that $\alpha$ is invertible. Then it is easy to transform the first column to $\begin{pmatrix}   1  \\ 0 
    \end{pmatrix}$ by elementary row operations. By the determinant condition the matrix is now of the form $\begin{pmatrix}   1 & * \\ 0 & 1 
    \end{pmatrix}$.
Since $\C[z]$ is a Euclidean ring, 
this shows that every matrix in $SL_n(\C[z])$, i.e., a matrix depending polynomially on one variable, is a product of finitely many unitriangular factors depending polynomially on that variable.

There is a generalization of this process by
Cohn \cite{MR0207856}*{Theorem 7.2} to the so-called weak Euclidean rings \cite{MR0153696}*{\S 2}. However, we don't know any application of this more general result for parametric dependence.

\subsection{Bass stable rank 1} \label{subsec:bsr1}

\begin{Def}
Let $R$ be a  unital commutative ring. An element
$(x_1, \dots, x_k) \in R^k$ is called \textsl{unimodular} if
\[
\sum_{j=1}^k x_j R =R.
\]
Let $U_k(R)$ the set of all unimodular elements in $R^k$.

An element $x=(x_1, \dots, x_{k+1}) \in U_{k+1}(R)$ is called \textsl{reducible}
if there exists $(y_1, \dots, y_k)\in R^k$ such that
\[
(x_1 + y_1 x_{k+1}, \dots, x_k + y_k x_{k+1}) \in U_k (R).
\]
The \textsl{Bass stable rank} of $R$, denoted by $\bsr(R)$  is the least $k\in\N$ such that every $x\in U_{k+1} (R)$ is reducible. If there is no such $k\in\N$,
then we set $\bsr(R) =\infty$.
\end{Def}

\begin{Rem}
 \label{r_bsr1}
The identity $\bsr(R) =1$ is equivalent to the following property:
For any $x_1, x_2\in R$ such that $x_1R + x_2R =R$,
there exists $y\in R$ such that $x_1 + y x_2 \in R^*$.
\end{Rem}

The Bass stable rank for algebras of holomorphic functions on finite dimensional Stein spaces was calculated by Alexander Brudnyi. Recall that the dimension of a Stein space is simply the dimension of its smooth part as a complex manifold.  
\begin{Thm}[Brudnyi \cite{MR3898325}*{Theorem 1.1}] \label{Bru-bsr}
Let $(X,\mathcal O_X)$ be a finite dimensional Stein space.  Then 
\[
\bsr (\mathcal O(X)) = \left\lfloor \frac 1 2\, {\rm dim}\, X \right\rfloor +1.
\]
\end{Thm}

\begin{Cor} \label{bsr1=dim1}
    The Bass stable rank for an algebra of holomorphic functions 
     on Stein space    is equal to one if and only if the dimension of the Stein space is one, i.e., in the case of open Riemann surfaces with possible singularities.
\end{Cor}

The following LULU decomposition of $SL_2(R)$ when $\bsr (R) =1$ is well known in algebraic $K$-theory. For a proof see Nikola\u{\i} Vavilov, Andrei Smolenski\u{\i} and Balasubramanian Sury \cite{MR2822515}*{Lemma 1}. 

\begin{Lem} \label{SL2-4-factor}
    Let\/ $R$ be a unital commutative ring of Bass stable rank one. Then
    $$ SL_2(R)=U^-(A_1,R) U(A_1,R) U^-(A_1,R) U(A_1,R), $$
    where $A_1$ denotes the rank one reduced irreducible root system of type $A$ and $U(A_1,R), U^-(A_1, R)$ are the upper and lower unitriangular matrices, respectively. 
\end{Lem}

\begin{Thm}[Michael Stein \cite{MR0528869}*{Theorem 2.2}] \label{Stein}
    Let $R$ be a unital commutative ring with Bass stable rank $\bsr(R) = m \ge 1$. Then
    \begin{enumerate}
        \item the map $SL_n(R)/E_n(R) \to SL_{n+1}(R)/E_{n+1}(R)$ induced by the canonical embedding of $SL_n(R)$ into $SL_{n+1}(R)$ is surjective for $n \ge m-1$;
        \item the map $Sp_{2n}(R)/Ep_{2n}(R) \to Sp_{2n+2}(R)/Ep_{2n+2}(R)$ induced by the canonical embedding of $Sp_{2n}(R)$ into $Sp_{2n+2}(R)$ is surjective for $n \ge \lfloor \frac{m}{2} \rfloor$.
    \end{enumerate}
\end{Thm}

When $R$ has Bass stable rank one, we conclude that $SL_n(R) = E_n(R)$ and $Sp_{2n}(R) = Ep_{2n}(R)$ for all $n \ge 2$. 
\medskip 

We recall some terminology from Lie algebra. Let $V$ be a finite dimensional real vector space endowed with a scalar product $(\cdot, \cdot)$. Every nonzero $\alpha \in V$ yields a reflection $w_\alpha \colon V \to V, \beta \mapsto \beta - \frac{2(\beta,\alpha)}{(\alpha,\alpha)}\alpha$. A subset $\Phi \subset V$ is called a {\it root system} of $V$ if the following conditions are satisfied:
\begin{enumerate}
    \item[(i)] $\Phi$ is finite, spans $V$ and does not contain $0$.
    \item[(ii)] $\Phi$ is invariant under reflection $w_\alpha$ for all $\alpha \in \Phi$.
    \item[(iii)] $2(\beta,\alpha)/(\alpha,\alpha)$ is an integer for all $\alpha, \beta \in \Phi$.
\end{enumerate}
The rank of $\Phi$ is the dimension of $V$. Moreover, a root system $\Phi$ is called 
\begin{itemize}
    \item reduced, if for each $\alpha \in \Phi$, the only roots proportional to $\alpha$ are $\alpha, -\alpha$. 
    \item irreducible, if it cannot be decomposed into two nonempty complementary subsets such that each root in one set is orthogonal to each root in the other.
\end{itemize}

The next result shows that the number of factors needed for the factorization does not increase with the size of the matrix. 

\begin{Thm}[O.\@ I.\@ Tavgen' \cite{MR1175793}*{Proposition 1}] \label{theorem:tavgen}
	Let $\Phi$ be a reduced irreducible root system of rank $l \geq 2$ and let $R$ be a unital commutative ring. Suppose that for subsystems $\Delta = \Delta_1, \Delta_l$ of rank $l-1$ the elementary Chevalley group $E(\Delta, R)$ admits a unitriangular factorization with $L$ factors
	\begin{align*}	
	E(\Delta, R) = U^-(\Delta, R)U(\Delta, R) \cdots U^{\pm}(\Delta, R).	
	\end{align*}
	Then the elementary Chevalley group $E(\Phi, R)$ admits a unitriangular factorization with the same number of factors
	\begin{align*}	
	    E(\Phi, R) =  U^-(\Phi, R)U(\Phi, R) \cdots U^{\pm}(\Phi, R).	
	\end{align*}
\end{Thm}

In the case of Bass stable rank one, combining Lemma \ref{SL2-4-factor} and Theorem \ref{theorem:tavgen} Vavilov, Smolenski\u{\i} and Sury obtain the following. 
\begin{Thm}[\cite{MR2822515}*{Theorem 1}]
    Let\/ $\Phi$ be a reduced irreducible root system and\/ $R$ be a unital commutative ring with $\bsr(R)=1$. Then the elementary Chevalley group\/ $E(\Phi,R)$ of type $\Phi$ admits unitriangular factorization
    $$ E(\Phi,R)= U^-(\Phi,R) U(\Phi,R) U^-(\Phi,R) U(\Phi,R) \quad \, \text{ of length } 4.
    $$
\end{Thm}

Together with Theorem \ref{Stein} we conclude

\begin{Thm} \label{bsr1-SL-Sp} 
    Let\/ $R$ be a unital commutative ring with $\bsr(R)=1$. Then the simply connected Chevalley groups of type $A_n, n \ge 1$ and type $C_n, n \ge 2$ admit unitriangular factorization
    \begin{align*}
        &SL_{n+1}(R) = U^-(A_{n},R) U(A_{n},R) U^-(A_{n},R) U(A_{n},R) \\
        &Sp_{2n}(R) = U^-(D_{n},R) U(D_{n},R) U^-(D_{n},R) U(D_{n},R) \quad \, \text{ of length } 4.
    \end{align*} 
\end{Thm}

The unitriangular factors for $Sp_{2n}$ are matrices of the following forms
\begin{align} \label{U-minus}
    U^{-}  \ni \begin{pmatrix} I_n & 0 \\ Z & I_n \end{pmatrix} \begin{pmatrix} A^{-T} & 0 \\ 0 & A \end{pmatrix} = \begin{pmatrix} A^{-T} & 0 \\ B & A \end{pmatrix}, \quad & Z = B A^T \text{ symmetric  }\\ 
     \text{and }& A \text{  upper unitriangular,} \nonumber
\end{align}
and the same for upper unitriangular
\begin{align} \label{U-plus}
     U^{+}  \ni \begin{pmatrix} I_n & Z \\ 0 & I_n \end{pmatrix} \begin{pmatrix} A^{-1} & 0 \\ 0 & A^T \end{pmatrix} = \begin{pmatrix} A^{-1} & C \\ 0 & A^{T} \end{pmatrix}, \quad & Z = C A^{-T} \text{ symmetric }\\ 
    \text{and }& A \text{ upper unitriangular.}  \nonumber
\end{align}
These factors are symplectic with respect to the symplectic form $\begin{pmatrix}
         0 & I_n \\
         -I_n & 0
     \end{pmatrix}$, but become unitriangular once we revert the order of the last $n$ basis elements of $R^{2n}$.
This is equivalent to working with another symplectic form
     $ \begin{pmatrix}
         0 & L_n \\
         -L_n & 0
     \end{pmatrix}$, where $L_n$ is the $n \times n$ matrix with $1$ along the skew-diagonal.
\bigskip

Here comes a list of examples of rings of complex-valued continuous or holomorphic functions with Bass stable rank 1, cf.\@ Raymond Mortini and Rudolf Rupp \cite{MR4390034}*{Example 26.106}.
For all these examples  factorization of symplectic and  special linear groups of arbitrary size into 4 elementary factors holds. 
\begin{enumerate}
    \item The disk algebra $A(\mathbb{D})$ of functions which are holomorphic in the unit disk $\mathbb{D}$ and continuous on $\overline{\mathbb{D}}$ (Peter Jones, Donald Marshall and Thomas Wolff \cite{MR0826488} and independently Gustavo Corach and Fernando Su{\' a}rez \cite{MR0806470}*{Theorem 1.2})
    \item $A^n(\mathbb{D})= \{f \in \O(\mathbb{D}): f^{(n)} \text{ continuously extendable to } \overline{\mathbb{D}} \}, \quad n \in \N$
    \smallskip
    \item $A^{\infty}(\mathbb{D}) = \bigcap_{n=1}^\infty A^n(\mathbb{D})$
    \smallskip
    \item The Hölder-Lipschitz class 
    $$\Lambda_\alpha(\overline{\mathbb{D}}) = \Bigl\{ f \in A(\mathbb{D}): \sup_{z,w \in \overline{\mathbb{D}}, z \neq w}\frac{\lvert f(z)-f(w)\rvert}{\lvert z-w \rvert^\alpha} < \infty  \Bigr\}, \quad 0 < \alpha \le 1 $$
    \item $\lambda_\alpha(\overline{\mathbb{D}}) = \{ f \in \Lambda_\alpha(\overline{\mathbb{D}}) : \lvert f(z)-f(w)\rvert = o(\lvert z-w \rvert^\alpha) \}, \quad 0 < \alpha < 1$
    \smallskip
    \item The Wiener algebra $W^+(\mathbb{D}) = \{ \sum_{n=0}^\infty a_n z^n : \sum_{n=0}^\infty \lvert a_n \rvert < \infty \} $
    \smallskip
    \item The Hardy algebra $H^\infty(\mathbb{D})$ of bounded holomorphic functions on $\mathbb{D}$ (Sergei Treil \cite{MR1183608})  and the Hardy algebra $H^\infty (X)$ for any connected open Riemann surface $X$ and for certain infinitely connected planar domains $X$ (Behrens domains) (Vadim Tolokonnikov \cite{MR1273527}) 
    \item $\C + B H^\infty(\mathbb{D})$ for a Blaschke product $B$ (Mortini, Amol Sasane and Brett Wick \cite{MR2610796})
    \smallskip
    \item  The algebra $A(\overline{X})$  of complex-valued functions which are continuous on $\overline{X}$ and holomorphic in $X$ for a bordered Riemann surface $X$ (Jürgen Leiterer \cite{MR4275099}*{Lemma 1.2})
    \item The algebra $\O(X)$ for open Riemann surface $X$ (Herta Florack \cite{MR0037362}) and for $X$ is an one-dimensional Stein space, i.e.\@ open Riemann surface with discrete singularities (Corollary \ref{bsr1=dim1})
    
\smallskip

\noindent The only example for continuous algebras we are aware of is the following:
    \item \label{example-vaserstein}  The algebra $\mathcal{C}(T)$  of complex-valued continuous functions on a normal topological space $T$ of real dimension one (Vaserstein \cite{MR0284476}*{Theorem 7})
\end{enumerate}

\subsection{Commutative Banach algebras} \label{subsec:banachalg}
In the list above we have several examples where the ring is a commutative Banach algebra with Bass stable rank one, their factorization is covered by Theorem \ref{bsr1-SL-Sp}. However, there are commutative Banach algebras with higher Bass stable rank, e.g.\@ $\mathcal{C}(T)$ has Bass stable rank $\lfloor \frac{d}{2}\rfloor + 1$ for a compact topological space $T$ with finite covering dimension $d$. 

\begin{Thm}[John Milnor \cite{MR0349811}*{Lemma 7.1}] \label{milnor}
    Let $R$ be a unital commutative Banach algebra over $\C$. Then the path component of the identity in $SL_n(R)$ coincides with the elementary group $E_n(R)$. 
\end{Thm}

Concerning the number of factors in the unitriangular factorization for $SL_n(R)$ over a commutative Banach algebra $R$, 
Brudyni \cite{MR3769716}*{Theorem 1.1} shows that there is a uniform bound  depending only on the size of the matrices and on the covering dimension of the maximal ideal space $M(R)$ endowed with the Gelfand topology. This bound is expressed in terms of the uniform bound in Vaserstein's theorem on continuous factorization, cf.\@ Theorem \ref{vaserstein}.

Bj{\" o}rn Ivarsson, Kutzschebauch and Erik L\o w proved the corresponding factorization for the symplectic group. 
\begin{Thm}[ 
\cite{MR4078081}*{Theorem 1}] \label{IKL-BA}
    Let $R$ be a unital commutative Banach algebra over $\C$. Then the path component of the identity in $Sp_{2n}(R)$ coincides with the elementary symplectic group $Ep_{2n}(R)$. 
\end{Thm}

\section{Nullhomotopy} \label{sec: nullhom}

Let $G = G(\C)$ be a linear algebraic group. 
The polynomial maps from an affine algebraic variety $X$ to $G$ and the holomorphic maps from a Stein space $X$ to $G$ are contained in the space of continuous maps from $X$ to $G$, thus they have an induced topology from the compact-open topology on the space of continuous maps. If such a map is a finite product of unitriangular matrices
\[
    f(x) = U^-_1(x) U^+_2(x) \cdots U^\pm_K(x),
\]
it is necessarily homotopic to the constant identity map. The reason is that every unipotent matrix is the sum of identity and a nilpotent matrix. For a nilpotent matrix $N$, 
\begin{align*}
    \log (I + N) = \sum_{k=1}^{\infty } \frac{(-1)^k}{k} N^k
\end{align*}
is a finite sum. Thus every unipotent matrix $U(x)$ can be written as the exponential of $\log U(x)$ and the homotopy is given by $\exp(t \log U(x))$. 
We say that a map $f \colon X \to G$ is \emph{nullhomotopic} if there is a continuous homotopy $h \colon [0,1] \times X \to G$ such that $h(0,x) = I$ and $h(1, x) = f(x)$ for all $x \in X$. 
\smallskip

In the continuous case, when the group $G$ is $SL_n(\C)$ or $Sp_{2n}(\C)$, the nullhomotopy characterizes unitriangular factorization, i.e.\@ every nullhomotopic continuous map from a finite dimensional normal topological space $T$ to $G$ can be factorized, cf.\@ Theorem \ref{vaserstein} and Theorem \ref{IKL-Sp-Cont}. In other words, a continuous map $T \to G$ can be factorized if and only if  it is contained  in the path component of identity of $\mathcal{C}(T, G)$
\begin{align*}
    E_n( \mathcal{C}(T) ) = (\mathcal{C}(T, SL_n(\C)))_0, \quad Ep_{2n}( \mathcal{C}(T) ) = (\mathcal{C}(T, Sp_{2n}(\C)))_0. 
\end{align*}
In particular, if a finite dimensional normal topological space $T$ is contractible with a contraction $\alpha \colon [0,1] \times X \to X$, every continuous map $f$ from $T$ to $G$ is connected to the constant map, hence is a product of unitriangular factors. Indeed the path $h \colon [0,1] \times X \to G, (t, x) \mapsto f(\alpha_t(x))$ connects $f$ to the constant map $f \circ \alpha_1$.  
\smallskip

Nullhomotopy also characterizes unitriangular factorization for holomorphic maps from a finite dimensional reduced Stein space $X$ to $SL_n(\C)$ or to $Sp_{2n}(\C)$, cf.\@ Theorems \ref{IK-SL}, Theorem \ref{Schott} and Theorem \ref{Bao-Sp}. 
By Grauert's Oka principle, the existence of a homotopy $h \colon [0,1] \times X \to G$ through continuous maps $h_t \in \mathcal{C}(X, G)$ with holomorphic endpoint maps implies the existence of a homotopy $h' \colon [0,1] \times X \to G$ through holomorphic maps $h'_t \in \O(X, G)$. Thus a holomorphic map $X \to G$ can be factorized if and only if  it is contained  in the path component of identity in $\mathcal{O}(X, G)$
\begin{align*}
    E_n( \mathcal{O}(X) ) = (\mathcal{O}(X, SL_n(\C)))_0, \quad Ep_{2n}( \mathcal{O}(X) ) = (\mathcal{O}(X, Sp_{2n}(\C)))_0. 
\end{align*}
We conclude again that for a contractible Stein space $X$ all holomorphic maps from $X$ to $SL_n(\C)$ or to $Sp_{2n}(\C)$ admit unitriangular factorization. 
\smallskip

For a commutative Banach algebra $R$ we know from Milnor's result (Theorem \ref{milnor}) that $E_n(R) = (SL_n(R))_0$, the elementary subgroup coincides with the path connected component of identity in $SL_n(R)$. 
Let us make a remark about path-connectedness of $SL_n (R)$. When we for example consider a Banach algebra $R$ of holomorphic functions on a  space $X$, the topological  contractibility of $X$ does not in general imply the path connectedness of $SL_n (R)$, as we had argued in the holomorphic and continuous cases. For example for  $f$ in the
Banach algebra $H^\infty (\mathbb{D})$ the map $[0,1] \to H^\infty(\mathbb{D}), t \mapsto f_t$, where $f_t(z) = f(tz)$ is not necessarily a continuous path in $H^\infty(\mathbb{D})$. 
On the other hand, the factorization in Theorem \ref{bsr1-SL-Sp} implies that $SL_n(H^\infty(\mathbb{D}))$ is path-connected. 
\begin{Cor}[to Theorem \ref{bsr1-SL-Sp} and Theorem \ref{IKL-BA}]
    For a commutative Banach algebra $R$ with Bass stable rank one, $SL_n(R)$ and $Sp_{2n}(R)$ are path-connected. 
\end{Cor}

In the algebraic case the continuous nullhomotopy is only a necessary condition. The main difference to the continuous and the holomorphic cases lies in maps to $SL_2$ as discovered by Paul Cohn. 

\begin{Lem}[Cohn \cite{MR0207856}*{\S 7}] \label{cohn}
    The matrix 
    $$\left(\begin{matrix} 1+zw & z^2 \\ -w^2 & 1-zw  \end{matrix}\right) $$
    is not a product of unitriangular matrices with entries in $\mathbb{C} [z, w]$. 
\end{Lem}

However for all other semisimple Chevalley groups $G$ there is a corresponding notion of nullhomotopy -- the so-called $\mathbb{A}^1$-nullhomotopy -- which characterizes unitriangular factorization of algebraic maps from smooth affine algebraic varieties to $G$, see Theorem \ref{St-E=G0}. 

\begin{Def}
    Let $f \colon X \to G$ be an algebraic map from an affine algebraic variety $X$ to a linear algebraic group $G$ over $\C$. 
    We say that $f$ is {\it $\mathbb{A}^1$-nullhomotopic} if there exists an algebraic map $h \colon \C \times X \to G$ such that $h_0 \equiv I$ and $h_1 = f$. 
\end{Def}

Clearly an $\mathbb{A}^1$-nullhomotopic map is nullhomotopic (in the topological sense).

Conversely topological nullhomotopy for an algebraic map between smooth affine algebraic varieties does not imply $\mathbb{A}^1$-nullhomotopy. 

The first smooth algebraic surfaces which are topologically contractible but not $A^1$-contractible were detected by Utsav Choudhury and Biman Roy  \cite{MR4698492}. In fact they showed that $A^1$-contractibility characterizes $\C^2$ among smooth  affine algebraic surfaces. We remind the reader that there are many smooth topologically contractible affine algebraic surfaces  not even homeomorphic and thus not isomorphic to  $\C^2$. They all have nontrivial fundamental group at infinity. By a famous result of Chidambaram Padmanabhan Ramanujam  $\C^2$ is characterized among smooth affine algebraic surfaces by being  homeomorphic to $\R^4$ \cite{MR0286801}.

\section{Algebraic results for polynomial rings} \label{sec: alg}

After Cohn's counterexample for factorization of matrices in $SL_2$ over polynomial ring, Andrei Suslin proved his famous result on factorization of matrices in $SL_n, n \ge 3$ over polynomial rings \cite[Corollary 6.7]{MR0472792} and Laurent polynomial rings \cite[Corollary 7.10]{MR0472792} of any number of variables.

Vyacheslav Kopeiko obtained the corresponding results for the symplectic group $Sp_{2n}$ for all $n \ge 2$ over polynomial rings \cite[Theorem 3.14]{MR0497932} and Laurent polynomial rings \cite[Theorem A]{MR1806868}.

The works of Anastasia Stavrova generalize these classical results and resolve the question about unitriangular factorization over polynomial rings and Laurent polynomial rings for all simply-connected Chevalley groups. We formulate the special case of her theorem over the complex numbers. 
\begin{Thm}[Stavrova \cite{MR3189425} Corollary 6.2] \label{thm:Laurentrings}
Let $G$ be a simply connected semisimple Chevalley group without $SL_2$ factor.
Then 
$$
    G(\C[x_1^{\pm 1},\ldots,x_n^{\pm 1},y_1,\ldots,y_m])= E(\C[x_1^{\pm 1},\ldots,x_n^{\pm 1},y_1,\ldots,y_m]) \, \text{ for any } m,n\ge 0.  
$$
\end{Thm}
Actually Stavrova's result holds in greater generality. The field $\C$ can be replaced by any principal ideal domain $D$ satisfying $G(D) = E(D)$ \cite[Theorem 1.2]{MR4855206}. An example for such a principal ideal domain is $D = \mathbb{Z}$. Therefore, Stavrova's theorem also includes the result of Fritz Grunewald, Jens Mennicke and Vaserstein \cite{MR1086811} about $Sp_{2n}(\mathbb{Z}[z_1, z_2, \dots, z_m])$ for all $m$ and $n \ge 2$. 
\medskip

The assumption of simply-connectedness in Theorem \ref{thm:Laurentrings} is not an essential restriction. There are also
results for non simply-connected Chevalley groups. For example the special orthogonal group $SO_{2n}$, whose simply-connected covering group is the spin group $Spin_{2n}$, has been considered by Suslin and Kopeiko. 
The surjective stability from Suslin--Kopeiko \cite{MR0469914}*{Theorem 7.8} for even orthogonal groups over polynomial rings together with the homotopy invariance in Max Karoubi \cite{MR0382400}*{Corollary 0.8} for regular Noetherian ring $R$ with $2 \in R^*$, yields 
\[
    SO_{2n}(R[x_1, \dots, x_m]) = SO_{2n}(R) EO_{2n}(R[x_1, \dots, x_m]) \text{ for } n \ge \max \{ 3, \dim R + 2 \}.
\]

A more general result which we again formulate over the complex numbers and use the fact that $G(k) = E(k)$ for an algebraically closed field $k$ is the following:
\begin{Thm}[Stavrova \cite{MR4048473}*{Corollary 1.4}]
    Let $G$ be a semisimple Chevalley group without $SL_2$ factor. Then 
    \[
        G(\C[x_1, \dots, x_m]) = E(\C[x_1, \dots, x_m]) \, \, \text{ for any } m \ge 0. 
    \]
\end{Thm}

By a further work of Stavrova \cite{MR3189425}, we see that
$\mathbb{A}^1$-nullhomotopy indeed characterizes unitriangular factorization of algebraic maps from smooth affine varieties to a linear algebraic group. 

\begin{Thm}[Stavrova \cite{MR3189425}*{Theorem 1.3, Lemma 3.3}] \label{St-E=G0}
    Let $G$ be a semisimple Chevalley group without $SL_2$ factor and $R$ a regular $\C$-algebra. Then
    \[
        E(R) = \{ g \in G(R): \exists \,\, h(x) \in G(R[x]) \text{ such that } h(0)=1, h(1)=g \}.
    \]
    In particular, every $\mathbb{A}^1$-nullhomotopic algebraic map from a smooth affine algebraic variety $X$ to $G(\C)$ is a product of unitriangular matrices with entries in $\C[X]$. 
\end{Thm}

To examplify the implications of this result to
the factorization problem in the algebraic case we state here
the following corollary. Recall that the Koras-Russel threefold is given as $KR= \{ (x,y,s,t) \in \C^4 : x+x^2y + s^2+t^3=0 \}$. It is diffeomorphic to $\R^6$ (hence topologically contractible), not algebraically isomorphic to $\C^3$, symplectomorphic to $\R^6$ and it is unknown whether it is biholomorphic to $\C^3$. Recently it was shown by Adrien Dubouloz and Jean Fasel \cite{MR3734108} that $KR$ is $A^1$-contractible. 

\begin{Cor}
  Let $G$ be a semisimple Chevalley group without $SL_2$ factor. Then any polynomial map from the Koras-Russell threefold $KR$ to $G(\C)$ is a product of polynomial unitriangular matrices. 
\end{Cor}

We conclude this section by a result of Wilberd van der Kallen \cite{MR0689383}.   Although every polynomial map $p \colon \C^m \to SL_n(\C)$ for fixed $n \ge 3, m \ge 1$ is a product of finitely many unitriangular matrices with polynomial entries, the number of factors needed is not uniformly bounded. 

\section{Continuous results} \label{sec: cont}

Vaserstein was a great promoter for the problem of factorization for continuous matrices. Jointly with William Thurston \cite{MR0855453}, they solved the first particular case by showing that every continuous map from $\R^3$ to $SL_n(\C)$ is a product of continuous unitriangular matrices. Later Vaserstein settled the problem for continuous maps from any finite dimensional topological space to $SL_n(\C)$.

\begin{Thm}[Vaserstein \cite{MR0947649}*{Theorem 4}] \label{vaserstein}
    For any natural number $n$ and an integer $d \ge 0$ there exists a natural
    number $L = L(n, d)$ such that for any finite dimensional normal topological space $T$ of dimension $d$, every nullhomotopic continuous map $f \colon T \to SL_n(\C)$ can be written as a product of no more than $L$ unitriangular matrices with entries in $\mathcal{C}(T)$. 
\end{Thm}

Contrary to the algebraic setting, there exists a uniform bound. However, we are not aware of any concrete value of $L(n,d)$ except $L(n,1)=4$ for all $n \ge 1$,  cf.\@ Theorem \ref{bsr1-SL-Sp} and Example \eqref{example-vaserstein} in Section \ref{subsec:bsr1}. It is worth mentioning that Theorem \ref{theorem:tavgen} implies 
\[
    L(n+1,d) \le L(n,d) \le L(2,d) \text{ for all } n \ge 2, d \ge 0. 
\]
Keith Dennis and Vaserstein \cite[Theorem 20]{MR0961333} also showed that $\lim_{n \to \infty} L(n,d) \le 6$. 

Inspired by the work of Brudnyi and Amol Sasane \cite{BrudSasa}, one can deduce a lower bound for $L(n,d)$. In fact using the proof of Theorem 5.3 (replacing elementary by unitriangular) in Christopher Phillips \cite{MR1305876} together with the proof of Lemma 2.1 in Brudnyi \cite{MR4357335}, one gets the following lower bound for the number $L$ in Theorem \ref{vaserstein}. Alternatively, one can also directly combine Theorem 2.3 of \cite{MR1305876} and Lemma 2.1 of \cite{MR4357335}.
\begin{Thm}  \label{lowerbound-sln}  
    \[
        L(n,d) \ge \max \left\{ 4, \left\lfloor \frac{d - 2}{n^2 -1} \right\rfloor - 1 \right\}.
    \]
\end{Thm}

\medskip

The following lower bound for the (at least) two dimensional case seems to be new. 
\begin{Lem}
    $L(2,d) \ge 5$ for all $d \ge 2$.
\end{Lem}
\begin{proof}
    We give a two dimensional example where four factors are not sufficient. Since
    $L(n,d+1) \ge L(n,d)$ the claim of the Lemma follows. The latter inequality follows  from the fact that 
     if a nullhomotopic continuous map $f \colon T \to SL_n(\C)$ cannot be written as a product of  $L$ unitriangular matrices with entries in $\mathcal{C}(T)$, then the map $F \colon T \times \R \to SL_n(\C)$ defined by $F(t,x) = f(t)$
     cannot be written as a product of  $L$ unitriangular matrices with entries in $\mathcal{C}(T \times \R)$.
    
    Before we describe the example, let us give some general consideration on factorization of continuous map into four factors. 
    Given a continuous map $f \colon X \to SL_2(\C)$, 
    suppose that there were continuous maps $g_1, g_2, g_3, g_4 \colon X \to \C$ such that
\begin{align*} 
    f= 
    \begin{pmatrix}    a & b\\ c & d    \end{pmatrix}
    = \begin{pmatrix}
        1 & 0\\ g_1 & 1
    \end{pmatrix}
    \begin{pmatrix}
        1 & g_2\\ 0 & 1
    \end{pmatrix}
    \begin{pmatrix}
        1 & 0\\ g_3 & 1
    \end{pmatrix}
    \begin{pmatrix}
        1 & g_4\\ 0 & 1
    \end{pmatrix}.
\end{align*}
Bring the first and the fourth factor to the left hand side, and carry out the multiplications 
\begin{align*} 
    \begin{pmatrix}
        a & b - a g_4\\ c - a g_1 & - g_4 (c - a g_1) + d - b g_1
    \end{pmatrix}
    =\begin{pmatrix}    1 + g_2 g_3 & g_2\\ g_3 & 1    \end{pmatrix}.
\end{align*}
In case $a \neq 0$, the first three equations read
\begin{align*}
    a =  1 + g_2 g_3,  \quad 
    g_4  =   \frac{1}{a} (b - g_2),  \quad 
    g_1  =   \frac{1}{a} (c - g_3), 
\end{align*}
and the fourth equation follows from the other three. If moreover $a \neq 1$,
\begin{align} \label{g_2}
    &\text{ any map } g_2 \colon \{ x \in X \mid a(x) \notin \{0,1\} \} \to \C^* \\ &\text{ is equivalent to a factorization in this part of } X. \nonumber
\end{align}

The fiber of the fibration $f^* \Phi_4$ (cf.\@ \eqref{PHI_K}) over $\{ x \in X \mid a(x) \notin \{0,1\} \}$ is $\C^{*}$, where
\begin{align*}
    \Phi_4 \colon \C^4 \to SL_2(\C), \,
    (z_1, z_2, z_3, z_4) \mapsto 
    \begin{pmatrix}
        1 & 0\\ z_1 & 1
    \end{pmatrix}
    \begin{pmatrix}
        1 & z_2\\ 0 & 1
    \end{pmatrix}
    \begin{pmatrix}
        1 & 0\\ z_3 & 1
    \end{pmatrix}
    \begin{pmatrix}
        1 & z_4\\ 0 & 1
    \end{pmatrix}.
\end{align*}
When $a = 1$, the fiber is the cross of axis $g_2 g_3 =0$. To get a factorization we must be able to extend $g_2$ to the whole $\mathbb{C}^2$ in a way that $g_1, g_3$ and $g_4$ are still well-defined.  

When $a = 0$, then the equations read
\begin{align*}
    1 + g_2 g_3 = 0, \quad g_2 = b, \quad g_3 = c, \quad 1 = - c g_4 + d - b g_1.
\end{align*}
Notice that $g_2$ and $g_3$ are prescribed as $b$ and $c$, respectively, and the fiber of $f^* \Phi_4$ here is $\C$ given by the last equation. 
\medskip

    Now we start to construct the example. Recall the holomorphic example from \cite{MR4678654}: Consider the following holomorphic mapping $f \colon \C^2 \to SL_2(\C)$
\begin{align*}
    f(z, w) = \begin{pmatrix}    (zw-1)(zw-2) & (zw-1)z + (zw-2)z^2 \\ h_1(z,w) & h_2 (z,w)   \end{pmatrix},
\end{align*}
where the functions in the second row are chosen such that $f(z,w)$ has determinant $1$. The existence of such polynomial functions follows from Hilbert's Nullstellensatz, or if one is looking for holomorphic functions from a standard application of Theorem B. For this observe that the functions in the first row have no common zeros.

In order to give an example on a two dimensional topological space, we restrict the example $f$ above to the following subset of $\C^2$
\[
    T = \{ (z,w) \in \C^2 : \exists \,\, t \in   [0,1]: zw = 2 - t, \lvert z \rvert = 1 \} \cong S^1 \times [0,1].
\]

Then on $\{ zw = 1, \lvert z \rvert = 1 \} \cong S^1 \times \{ 1 \}$, $g_2(z,w) = -z^2$ and on $\{ zw = 2, \lvert z \rvert = 1 \} \cong S^1 \times \{ 0 \}$, $g_2(z,w) = z$. Since on $ S^1 \times (0,1)$ $a$ is neither $0$ nor $1$, by \eqref{g_2} $g_2 \colon S^1 \times [0,1] \to \C^*$ induces a family of continuous self-maps of $S^1$
\begin{align*}
    F \colon S^1 \times [0,1]  \to S^1, (\theta, t) \mapsto g_2(\theta, (2-t)/\theta) / \lvert g_2(\theta, (2-t)/\theta)\rvert, 
\end{align*}
connecting  between $F(\theta,0) = \theta$ and $F(\theta,1)= -\theta^2$. Since these two self-maps of $S^1$ have different degrees, we find a contradiction. 
\end{proof}
\smallskip

\noindent {\bf Example}
The matrix of Cohn (Lemma \ref{cohn}), which cannot be algebraically factorized, can be factorized as a product of four continuous unitriangular matrices. 
First note that we can build a continuous section on the complement of $zw = 0$
\begin{align} \label{D=0}
    g_1^{(0)} = -w/z,\quad  g_2^{(0)} = z^2, \quad  g_3^{(0)} = -w^2, \quad  g_4^{(0)} = 0, 
\end{align}
while on the complement of $zw = -1$ we write down the continuous section
\begin{align} \label{D=-1}
    &g_1^{(-1)} = -w^2 \frac{1 + \lvert w \rvert^{-3/2}}{1 + zw}, \quad g_2^{(-1)} = \frac{z}{w}\lvert w \rvert^{3/2}, \quad \\ & g_3^{(-1)} = w^2 \lvert w \rvert^{-3/2}, \quad g_4^{(-1)} = z \frac{z - \lvert w \rvert^{3/2}/w}{1 + zw}. \nonumber
\end{align}
Based on the observation \eqref{g_2} above, for a fixed $D=a-1 \in \C \setminus \{ -1, 0 \}$, the restriction of any continuous section to $zw = D$ is a continuous map from $\C^*$ to $\C^*$ (given by $g_2$), where we choose $z$ as the parameter on $\{ zw = D \} \cong \C^*$. The restrictions of the two sections above to $zw = D$ both have degree $2$. Therefore $g_2^{(0)}$ can be deformed to $g_2^{(-1)}$ by homotopy for each fixed $D$. In order to build a global section, we will glue these two sections together on the overlap $U = \{ -0.7 < \Re(zw) < -0.3 \}$ of their domains of definition. Since on this domain the real part of $\lvert w \rvert^{3/2} /(zw)$ is always positive, a deformation retraction $c$ of the right half plane to the point $1$ gives a continuous homotopy 
\[
    h \colon [0,1] \times U \to \C^*, (t, z, w) \mapsto z^2 c(t, \lvert w \rvert^{3/2} /(zw))
\]
between the two sections on $U$. We take a cutoff function $\chi \colon \R \to [0,1]$ such that $\chi \equiv 0$ on $(- \infty, -0.5]$ and $\chi \equiv 1$ on $[-0.4, \infty)$. Then the composition $h( \chi(\Re(zw),z,w) \colon U \to \C^*$ gives a continuous gluing of the two sections \eqref{D=0} and \eqref{D=-1}, hence a global factorization. 

Let us now explain why the example of Cohn cannot be factorized as four holomorphic unitriangular matrices. Suppose that such a factorization exists. On $D \neq -1$ we have $g_2 g_3 = zw$, which forces $g_2$ to be of the form either $g_2 = e^{f(z,w)}$ nowhere vanishing, or $g_2 = z e^{f(z,w)}$ or $g_2 = w e^{f(z,w)}$ or $g_2 = zw e^{f(z,w)}$. Thus for a fixed $D \in \C \setminus \{ -1, 0 \}$ the degree of the map $\{ zw = D \} \to \C^*$ is 0, 1 or $-1$. Since $g_2 = z^2$ on $D = -1$, we conclude that for $\lvert D+1 \rvert < \varepsilon$ the degree of the map $\{ zw = D \} \to \C^*$ would be $2$. This
contradicts the existence of a global section, thus a factorization into 4 holomorphic factors.

\begin{Rem}
    The continuous section that we constructed does not avoid the singularity set of $f^* \Phi_4$, namely $g_2$ and $g_3$ are both zero over $\{ w = 0 \} \subset \{ a = 1 \}$, meaning the section passes through the singularity of the fiber, the zero point of the cross of axis $g_2 g_3 =0$. In fact, any continuous factorization into four factors has to meet the singularity set of the map $f^* \Phi_4$. Indeed, if it could avoid the singularity set, the Oka principle for  stratified elliptic submersions would provide us a holomorphic factorization into 4 factors. 
\end{Rem}

\smallskip

The corresponding problem for continuous maps from any finite dimensional topological space to $Sp_{2n}(\C)$ has been settled by 
Ivarsson, Kutzschebauch and L\o w. 
\begin{Thm}[\cite{MR4078081}*{Theorem 3}]\label{IKL-Sp-Cont}
    For any natural number $n$ and an integer $d \ge 0$ there exists a natural
    number $L_{sp} = L_{sp}(n, d)$ such that for any finite dimensional normal topological space $T$ of dimension $d$, every nullhomotopic continuous map $f \colon T \to Sp_{2n}(\C)$ can be written as a product of no more than $L_{sp}$ matrices of the form $\begin{pmatrix} I_n & 0 \\ C & I_n \end{pmatrix}$ and $ \begin{pmatrix} I_n & B \\ 0 & I_n \end{pmatrix}$ where $B,C$ are $n \times n$ symmetric matrices with entries in $\mathcal{C}(T)$.
\end{Thm}
Since the matrices used in Theorem \ref{IKL-Sp-Cont} are special types of \eqref{U-minus} and \eqref{U-plus}, we see that every nullhomotopic element of $Sp_{2n}(\mathcal{C}(T))$ lies in $Ep_{2n}(\mathcal{C}(T))$. 
\smallskip

For the number of factors, if the Bass stable rank is one, $L_{sp}(n,1)=4$ for all $n \ge 1$ as in the $SL_n$ case. Theorem \ref{theorem:tavgen} implies 
\[
    L_{sp}(n+1,d) \le L_{sp}(n,d) \le L_{sp}(1,d) =  L(2,d) \text{ for all } n \ge 1, d \ge 0. 
\]
Furthermore, $L_{sp}(n,d)$ has an asymptotical upper bound for sufficiently large $n$ as in the $SL_n$ case. We will show this in a forthcoming paper. 

As in Theorem \ref{lowerbound-sln}, one gets the following lower bound for $L_{sp}$. 
\begin{Thm}
    \[
        L_{sp}(n,d) \ge \max \left\{ 4, \left\lfloor \frac{d - 2}{4n^2 -1} \right\rfloor - 1 \right\}.
    \]
\end{Thm}

\section{Holomorphic results}  \label{sec :holo}

Manfred Klein and Karl Ramspott \cite[\S IV]{MR0966022} showed that every holomorphic map from a noncompact Riemann surface to $SL_n(\C)$ is a product of unitriangular factors. 
Gromov \cite{MR1001851} asked the following more general question: Does every holomorphic map $\C^m \to SL_n(\C)$ decompose into a finite product of holomorphic maps sending $\C^m$ into unipotent subgroups in $SL_n(\C)$?
He called it after the advocate of the continuous factorization {\it the Vaserstein Problem}. It was solved in greater generality by Ivarsson and Kutzschebauch. 

\begin{Thm}[Ivarsson--Kutzschebauch \cite{MR2874639}]\label{IK-SL}
    There exists a natural
    number $K = K(n, d)$ such that given any finite dimensional reduced Stein space $X$ of dimension $d$, every nullhomotopic holomorphic map $f \colon X \to SL_n(\C)$ can be written as a product of no more than $K$ unitriangular matrices with entries in $\O(X)$. 
\end{Thm}
Regarding the number of factors, first we have the Bass stable rank one case $K(n,1)=4$ for all $n \ge 1$,  cf.\@ Theorem \ref{Bru-bsr} and Theorem \ref{bsr1-SL-Sp} in Section \ref{subsec:bsr1}. Again Theorem \ref{theorem:tavgen} implies 
\[
    K(n,d) \le K(2,d) \text{ for all } n \ge 2, \, d \ge 0 
\]
and $\lim_{n \to \infty} K(n,d) \le 6$ from Dennis--Vaserstein \cite[Theorem 20]{MR0961333}. 

From Phillips' lower bound for continuous factorization into exponential factors \cite{MR1305876}, Brudnyi and  Sasane recently deduced the following lower bound for unitriangular factorization. 

\begin{Thm}[\cite{BrudSasa} Remark 3.8]
    \[
        K(n,d) \ge \max \left\{  4, \left\lfloor \frac{d}{n^2-1} \right\rfloor -3   \right\}. 
    \]
\end{Thm}

Ivarsson and Kutzschebauch \cite{MR2869067} showed that $K(2,2) = 5$. As an example let us look at the Cohn matrix $\left(\begin{matrix} 1+zw & z^2 \\ -w^2 & 1-zw  \end{matrix}\right)$. First we make the $(1,1)$ entry invertible
\[
    \left(\begin{matrix} 1+zw & z^2 \\ -w^2 & 1-zw  \end{matrix}\right) \begin{pmatrix}
        1 & 0 \\ \frac{e^{zw} - 1 - zw}{z^2} & 1
    \end{pmatrix}
    = 
    \begin{pmatrix}
        e^{zw} & z^2 \\ * & 1 -zw
    \end{pmatrix}.
\]
Now since  $e^{zw}$ is invertible one can proceed like over a field and reduce the right hand side by 4 more column operations to the identity matrix. This yields the desired factorization into 5 holomorphic factors.

\medskip

The proof of Theorem \ref{IK-SL} goes by an induction on the size $n$ of the matrices. The inductive step is based on the fact that the restriction of $\Phi_K$,
composed with the projection to the last row, 
to its regular points is a stratified elliptic submersion. Here $\Phi_K$ is given by 
\begin{align}\label{PHI_K}
    \Phi_K \colon \left( \C^{\frac{n(n-1)}{2}} \right)^K &\to SL_n(\C), \\ 
    (Z_1, \dots, Z_K) &\mapsto \left(\begin{matrix} 1 & 0 \cr Z_1 & 1 \cr \end{matrix} \right)   
\left(\begin{matrix} 1 & Z_2 \cr 0 & 1 \cr \end{matrix} \right)  \ldots \left(\begin{matrix} 1 & Z_K\cr 0 & 1 \cr \end{matrix} \right). \nonumber
\end{align}
The Oka principle for stratified elliptic submersions due to Forstneri{\v c} implies  that a global continuous section of the pullback $f^* \Phi_K $ by a holomorphic map $f \colon X \to SL_n(\C)$ from a Stein space $X$  can be deformed to a global holomorphic section. Such a global holomorphic section is exactly a product of holomorphic unitriangular matrices whose last row coincides with the last row of $f$.  The existence of a global continuous section
follows from Vaserstein's result, Theorem \ref{vaserstein}.

Using a similar strategy, the symplectic version of Vaserstein's problem was solved by Josua Schott.

\begin{Thm}[General case by Schott \cite{Schott}, Ivarsson--Kutzschebauch--L\o w \cite{MR4578529} for $n=2$]\label{Schott}
There exists a natural
number $K'_{sp} = K'_{sp}(n, d)$ such that given any finite dimensional reduced Stein space $X$ of dimension
$d$, every nullhomotopic holomorphic map $f \colon X \to \mathrm{Sp}_{2n}(\mathbb{C})$ can be written as a product of no more than $K'_{sp}$ matrices of the form $\begin{pmatrix} I_n & 0 \\ C & I_n \end{pmatrix}$ and $ \begin{pmatrix} I_n & B \\ 0 & I_n \end{pmatrix}$ where $B,C$ are $n \times n$ symmetric matrices with entries in $\O(X)$. 
\end{Thm}
The proof of the fact that the corresponding map $\Phi_K$, composed with the projection to the last row, is a stratified elliptic submersion is much more difficult than in the $SL_n$ case. The number of factors needed for this factorization is mostly unknown even over fields, cf.\@ Pengzhan Jin, Zhangli Lin and Bo Xiao \cite{MR4442601}. 

Actually the situation becomes more pleasant if one allows more general factors of the form \eqref{U-minus} and \eqref{U-plus}. Indeed by Tavgen's trick, see Theorem \ref{theorem:tavgen}, the number of unitriangular factors needed for $Sp_{2n}$ is then bounded by the correponding number for $Sp_2 = SL_2$. In addition, the proof of the stratified ellipticity of the corresponding submersion can be simplified.

\begin{Thm}[Huang--Kutzschebauch--Bao Tran \cite{Symplectic-Bao}]\label{Bao-Sp} 
There exists a natural
number $K_{sp} = K_{sp}(n, d)$ such that given any finite dimensional reduced Stein space $X$ of dimension
$d$, every nullhomotopic holomorphic map $f \colon  X \to Sp_{2n}(\mathbb{C})$ can be written as a product of no more than $K_{sp}$ matrices of the form \eqref{U-minus} and \eqref{U-plus} with entries in $\O(X)$.
\end{Thm}

Concerning the number of factors, when the Bass stable rank is one, $K_{sp}(n,1)=4$ for all $n \ge 1$ as before. Also Theorem \ref{theorem:tavgen} implies 
\[
    K_{sp}(n,d) \le K_{sp}(1,d) = K(2,d) \text{ for all } n \ge 1, \, d \ge 0. 
\]

Analogous to the $SL_n$ case, Corollary 13.6 in Brudnyi--Sasane \cite{BrudSasa} and Proposition 4.1 of \cite{MR4678654} imply the following lower bound for unitriangular factorization of symplectic matrices.
\begin{Thm}
    \begin{align*}
        K_{sp}(n,d)  \ge  \max \left\{ 4, \left\lfloor \frac{d}{4n^2-1} \right\rfloor -3 \right\}, \quad n \ge 1.
    \end{align*}
\end{Thm}

\smallskip

Obviously Schott's factorization implies Theorem \ref{Bao-Sp}. Conversely by the following comparison between the numbers of factors needed, Theorem \ref{Bao-Sp} implies Schott's result.

\begin{Thm}[\cite{Symplectic-Bao}]\label{Omega-to-OmegaTilde}
For $n \ge 4$ we have
$$6 \le K'_{sp}(n,d) \le 7  K_{sp}(n,d).$$
For matrices of smaller sizes, there is a better upper bound
\begin{align*} 
    K'_{sp}(n,d) \le 4  K_{sp}(n,d), \quad \text{ for } n = 2 \text{ and } n=3.
\end{align*}
\end{Thm}

Moreover, combining Theorem \ref{Omega-to-OmegaTilde}
with Lemma \ref{SL2-4-factor}, Theorem \ref{theorem:tavgen} and \cite[Theorem 1]{MR2869067} (cf.\@ \cite[Theorem 3.1]{MR4678654}),
in certain cases we are able to give explicit bounds for the number of factors in Schott's theorem as follows:

\begin{Prop}[\cite{Symplectic-Bao}] \label{bao-bounds}
    The number of factors for $N \ge n+1$ is bounded by the number for $n$:
    \begin{align*} \label{K(N,d)}
        K'_{sp}(N,d) \le 7 K_{sp}(n,d).  
    \end{align*}
    When the Stein space $X$ is one-dimensional,
    \begin{align*}
        5 &\le K'_{sp} (n,1)  \le 16 \ \  \text{for} \ \ n=2,3,  \\
   6 &\le K'_{sp} (n,1)  \le 28 \ \ \text{for } \  n \ge 4.  
    \end{align*}
    When the Stein space $X$ is two-dimensional,
    \begin{align*}   
   5 &\le K'_{sp} (n,2)  \le 20 \ \ \text{for} \ \ n=2,3, \\
   6 &\le   K'_{sp}(n,2) \le  35  \ \  \text{for } \ n \ge 4. 
  \end{align*}
\end{Prop}

\section{Vector bundle automorphisms} \label{sec: VB}

We consider a holomorphic vector bundle $E \to X$ over a Stein space $X$.
By $E_x$ we denote the fiber of the bundle $E$ over the point $x \in X$.
Let $F \colon X \to \SAut(E)$ be a special holomorphic vector bundle automorphism of $E$. Here special refers to the fact that the holomorphic function $\det F \colon X \to \C^\ast$ is constantly $1$. Recall from \cite{MR0098196} that $F$ is a global section in the automorphism bundle $\Aut(E)$, which is a holomorphic
fibre bundle with typical fibre $GL_n(\C)$, and it is a group bundle in the sense of Cartan \cite{MR0098196}. The condition that the determinant of $F$ (recall that the determinant of an endomorphism of a vector space is defined independently of the choice of a basis of the vector space) can be formulated as $F$ being a section in the subbundle $\SAut(E) \subset \Aut(E)$, which is again a  group bundle, where the fibre $GL_n(\C)$ has been replaced by $SL_n(\C)$, i.e., the transition functions (acting by conjugation, i.e.\@, inner automorphisms of $GL_n(\C)$) remain the same. 

We call an automorphism of $E$, $\alpha \in \Gamma(X,\Aut(E))$, \textit{unipotent} if and only if $\alpha - Id$ is nilpotent, viewed as an endomorphism of the fibres of $E$, i.e., for each point $x$, the linear map $\alpha (x) - Id_{E_x} \colon E_x \to E_x$ is nilpotent. Clearly such $\alpha$ is necessarily of determinant $1$, i.e., $\alpha \in \Gamma (X, \SAut(E))$. We denote the subset of unipotent global holomorphic sections  by $\U(E)$. 

The previous continuous and holomorphic cases correspond to trivial vector bundles $X \times \C^r$, where the special automorphism bundle $\SAut(E) \cong X \times SL_r(\C)$. The sections of $\SAut(E)$ are simply continuous or holomorphic maps from $X$ to $SL_r(\C)$. We have the canonical set of constant holomorphic sections of the form
\[
    \{ I_r + \C e_{ij}, i \neq j \},
\]
which generates $SL_r(E_x)$ over each point $x \in X$. Here $e_{ij}$ is the elementary matrix with $1$ at the $(i,j)$ entry and 0 else where. Our holomorphic factorization problem, Theorem \ref{IK-SL}, is equivalent to writing any holomorphic section $F$ of $\SAut(E)$ as a product 
\[
    F(x) = \prod_{k=1}^K \left(  \prod_{i \neq j} (I_r + f_{ij}^{(k)}(x) e_{ij} ) \right) \text{ for all } x \in X,
\]
with holomorphic functions $f_{ij}^{(k)} \in \O(X)$. 

For a nontrivial vector bundle $E \to X$, one first needs to find a set of global holomorphic nilpotent sections $N_i, i = 1, \dots, L$ of $\End(E) \to X$ such that 
\begin{align*}
    \{ Id_{E_x} + z N_i(x) : z \in \C, \, i = 1, \dots, L \} \text{ generate } SL_r(E_x) \text{ for all } x \in X.
\end{align*}

Based on this set of globally generating unipotent automorphisms, the goal is to factor a given special automorphism, a section $F \colon X \to \SAut(E)$, in the following manner: Find coefficient holomorphic functions $f_{ij}$ on $X$ such that 
\begin{align} \label{F-SAut}
    F(x) = \prod_{j=1}^K \left(  \prod_{i=1}^L (Id_{E_x} + f_{ij}(x) N_i(x) ) \right) \text{ for all } x \in X. 
\end{align}
The construction of the set of $N_i$ and the solution of this factorization problem has been worked out for rank 2 bundles by George Ioni\c t\u a and Kutzschebauch, leading to the following result.

\begin{Thm}[\cite{MR4840978}]\label{specialVB} Let $X$ be a reduced finite dimensional Stein space and $E \to X$ a rank $2$ holomorphic vector bundle over $X$. Then  a holomorphic section $F$ of $\SAut(E)$ is a product of unipotent holomorphic sections $u_i \in \U (E)$, $i = 1, 2, \ldots, K$,
$$F(x) = u_1 (x) \cdot u_2 (x) \cdot \ldots \cdot u_K (x) \, \text{ for all } x \in X
$$
if and only if $F$ is nullhomotopic.
\end{Thm}

The proof uses again the Oka principle by Forstneri{\v c}, and moreover the Oka principle by Otto Forster and Karl Ramspott \cite{MR0212211}. As input for the Oka principle, one first has to solve \eqref{F-SAut} with continuous $f_{ij}$'s. This can be seen as a version of Vaserstein's topological factorization over the trivial bundle. 
\medskip

We would like to remark that in the continuous category 
Jakob Hultgren and Erlent F.\@ Wold \cite{MR4236641} proved unipotent factorizations of automorphisms of complex vector bundles over finite dimensional locally finite CW-complexes.

\begin{Thm}[\cite{MR4236641} Theorem 1]\label{Hultgren-Wold}
Let $X$ be a locally finite, finite dimensional CW-complex and $E \rightarrow X$ be a complex vector bundle 
of rank $r \ge 2$. 
Let $F$ be a nullhomotopic continuous special vector bundle automorphism of $E$, then there exist unipotent continuous vector bundle automorphisms
$u_1,\ldots, u_K$ such that 
$$
F = u_1 \circ  u_2 \circ \cdots \circ u_K.
$$
\end{Thm}

Note that their solution is different from being a continuous solution to \eqref{F-SAut}, and cannot be used as an input for an Oka principle. In fact, the unipotent vector bundle automorphisms $u_i$ are constantly identity on open subsets of $X$. 

There is also a symplectic result by the same authors.

\begin{Thm}[\cite{MR4236641} Theorem 3]\label{HultgrenWold-Sp}
Let $X$ be a locally finite, finite dimensional CW-complex  and $E \rightarrow X$ be a complex symplectic vector bundle over $X$. Let $F$
be a nullhomotopic continuous symplectic vector bundle automorphism of $E$, then there exist unipotent continuous symplectic vector bundle automorphisms
$u_1, \ldots, u_K$ such that 
$$
F= u_1 \circ u_2 \circ \cdots \circ u_K.
$$
\end{Thm}
In fact, they proved this theorem under the additional assumption that the structure group $Sp_{2n}(\C)$ can be reduced to the maximal compact subgroup. By a classical result see Steenrod \cite[\S 12.5]{MR1688579}, this reduction is always possible: If $G$ is a Lie group and $H$ a closed subgroup such that $G/H$ is contractible, then every fiber bundle over $X$ with structure group $G$ is equivalent to a fiber bundle with structure group $H$.

\section{Future developments, open problems}

As a natural extension, we propose to study the analogues of Theorems \ref{IK-SL}, \ref{Schott} and \ref{Bao-Sp} for other simple complex Lie groups, namely the orthogonal groups and the exceptional groups. 

It is also of interest to have more accurate estimates for the number of factors for the various factorizations  discussed in this survey. 

\begin{OP}
    Determine the optimal numbers of factors $L(n,d)$ in Theorem \ref{vaserstein}, $L_{sp}(n,d)$ in Theorem \ref{IKL-Sp-Cont}, $K(n,d)$ in Theorem \ref{IK-SL}, $K'_{sp}(n,d)$ in Theorem \ref{Schott} and $K_{sp}(n,d)$ in Theorem \ref{Bao-Sp}. 
\end{OP}
This is a very ambitious problem, any estimates 
for those numbers for rings of continuous or holomorphic functions on spaces of dimension $3$ or higher would be a success.

\begin{OP}
    Generalize Theorem \ref{specialVB} on rank 2 bundles to any rank $r \ge 2$ and potentially with additional geometric structures, e.g., the structure group being the complex symplectic group $Sp_{2n}(\C)$ or the complex orthogonal group $O_n(\C)$. 
\end{OP}

\medskip

Given a parabolic subgroup $P$ in a Chevalley group $G$, the elementary group $E_P(R)$ is generated by the unipotent radical $U_P(R)$ of $P$ and the unipotent radical $U_{P^-}(R)$ of an opposite parabolic subgroup $P^-$. By a result of Victor Petrov and Stavrova \cite[Theorem 1]{MR2473747}, the elementary group $E_P(R)$ associated to $P$ coincides with the elementary group $E(R)$ associated to the Borel subgroup $B$ contained in $P$. This leads us to 

\begin{OP} \label{parabolic}
    Determine the bounds for number of elementary factors in $U_P(R)$ and $U_{P^-}(R)$ for different choices of $P$ for rings $R$ of continuous or holomorphic functions. 
\end{OP}
For example, if we look at $(n k \times nk )$-matrices of determinant 1 as $(n\times n)$-matrices of determinant 1 with entries being $(k\times k)$-matrices:
$$SL_{nk} (\C) = SL_n (M_k (\C)).$$
Then an example of  a parabolic subgroup $P$ of $SL_{nk} (\C)$ is given by
$$P= \{ A = (A_{ij})_{i,j =1}^n  \in SL_n (M_k (\C)) : A_{ij}=\mathbf{0} \quad \forall i<j \}.$$ 
Theorem 1 of \cite{MR2473747} together with
Theorem \ref{IK-SL} imply for any Stein space $X$ the equalities
$$
    SL_n (M_k (\mathcal{O} (X)))_0= E_n  (M_k (\mathcal{O} (X)) =E_{nk} (\mathcal{O} (X))=  SL_{nk} (\mathcal{O} (X))_0.
$$  
Problem \ref{parabolic}
in this concrete case is asking about the number of factors in factorizing a nullhomotopy holomorphic  $(nk\times nk)$-matrix of determinant 1 into factors having
1's on the diagonal and 0's not only  
below the diagonal but also above the diagonal in the $k\times k$-blocks above the diagonal.

\medskip

As future development, it is also worth to study the commutator width of the path component of $G(R)$ for  rings $R$ of continuous or holomorphic functions. By Corollary 14 in Dennis--Vaserstein \cite{MR0961333}, the commutator width $c(E(R))$ of an elementary group $E(R)$ is related to the number $K$ of unitriangular factors by
\begin{align*}
    c( E(R) ) \le \left\lfloor \frac{K}{2} \right\rfloor + 3.
\end{align*}



\end{document}